\documentclass{article}%
\usepackage{amsmath}
\usepackage{amsfonts}
\usepackage{amssymb}
\usepackage{graphicx}%
\providecommand{\U}[1]{\protect\rule{.1in}{.1in}}
\newtheorem{theorem}{Theorem}

\newtheorem{corollary}[theorem]{Corollary}

\newtheorem{example}[theorem]{Example}

\newtheorem{lemma}[theorem]{Lemma}

\newtheorem{proposition}[theorem]{Proposition}
\newtheorem{remark}[theorem]{Remark}

\newenvironment{proof}[1][Proof]{\noindent\textbf{#1.} }{\ \rule{0.5em}{0.5em}}
\begin{document}

\begin{center}
{\Huge A Galois correspondence for compact group actions on C*-algebras}

\bigskip

Costel Peligrad\bigskip
\end{center}

Department of Mathematical Sciences, University of Cincinnati, PO Box 210025,
Cincinnati, OH 45221-0025. Email: peligrc@ucmail.uc.edu \vskip10pt

\bigskip

\textbf{Abstract. }In this paper, we prove a Galois\textbf{ }correspondence
for compact group actions on C*-algebras in the presence of a commuting
minimal action. Namely, we show that there is a one to one correspondence
between the C*-subalgebras that are globally invariant under the compact
action and the commuting minimal action, that in addition contain the fixed
point algebra of the compact action and the closed, normal subgroups of the
compact group.

\bigskip



\section{\bigskip Introduction}

Let $(M,G,\delta)$ be a W*-dynamical system with $M$ a von Neumann algebra,
$G$ a compact group and $\delta:g\in G\rightarrow\delta_{g}\in
Aut(M\mathcal{)}$ a homomorphism of $G$ into the group $Aut(M\mathcal{)}$ of
all automorphisms of $M$ such that the mapping $g\rightarrow\delta_{g}(m)$ is
simple-weakly continuous for every $m\in M$. Denote by $Aut_{\delta}(M)$ the
subgroup of $Aut(M\mathcal{)}$ consisting of all automorphisms commuting with
all $\delta_{g},g\in G$. In [7], (see also [[15], Theorem 3]) it is proven
that if $Aut_{\delta}(M)$ contains an ergodic subgroup $\mathcal{S}$, then
there is a one to one correspondence between the set of normal, closed
subgroups of $G$ and the set of all $G$ and $\mathcal{S}$ globally invariant
von Neumann subalgebras $N$ with $M^{G}\subset N\subset M$. This
correspondence is given by: $N\longleftrightarrow G^{N}$ where $G^{N}=\left\{
g\in G|\delta_{g}(n)=n,n\in N\right\}  $. The main technical tool in
Kishimoto's approach is the method of Hilbert spaces inside a von Neumann
algebra, as developed in [14]. Later, in [6], the case of irreducible actions
was considered. They proved that if $M^{G}\subset M$ is an irreducible pair of
factors, i.e. $(M^{G})^{\prime}\cap M=\mathcal{C}I$, then there is a one to
one correspondence between the set of intermediate subfactors $M^{G}\subset
N\subset M$ and the closed subgroups of $G$ given by $N\longleftrightarrow
G^{N}$, where, as above, $G^{N}=\left\{  g\in G|\delta_{g}(n)=n,n\in
N\right\}  $. This correspondence is called Galois correspondence. In [6], an
action of a compact group with the property $(M^{G})^{\prime}\cap
M=\mathcal{C}I$ is called minimal. Notice that in this case, $\mathcal{S}%
=\left\{  Ad(u)|u\in M^{G},\text{unitary}\right\}  $ is ergodic on $M$ and
that $N$ is obviously $\mathcal{S}$-invariant but is not required to be
$G-$invariant. This result was extended to the case of compact quantum group
actions on von Neumann factors by Tomatsu [16]. Both papers, [6] and [16] make
extensive use of the method of Hilbert spaces inside a von Neumann algebra and
other methods specific for von Neumann algebras. In this paper, we will prove
a result that extends Kishimoto's result to the case of compact actions on
C*-algebras commuting with minimal actions, as defined below. To the best of
our knowledge, this is the first result of this kind for C*-dynamical systems.
The notion of minimal action that will be used in this paper is different from
the one used in [6]. Our methods are specific to C*-dynamical systems and
give, in paticular, a new proof of Kishimoto's result. We also give an example
that shows that the result is not true if the commutant of the compact action
satisfies a weaker ergodicity condition, that, in the case of von Neumann
algebras is equivalent with the usual one.

\section{Notations and preliminary results}

\subsection{Ergodic actions on C*-algebras}

If $M$ is a von Neumann algebra, a subgroup $\mathcal{S\subset}Aut(M)$ is
called ergodic if $M^{\mathcal{S}}=\mathcal{C}I$, where $\mathcal{C}$ is the
set of complex numbers and $M^{\mathcal{S}}$ denotes the fixed point algebra,
$M^{\mathcal{S}}=\left\{  m\in M|s(m)=m,s\in\mathcal{S}\right\}  $. In the
case of C*-algebras there are several distinct notions of ergodicity that are
all equivalent for von Neumann algebras. These notions are distinct even for
abelian C*-algebras, the case of topological dynamics. Let $A$ be a C*-algebra
and $\mathcal{S}\subset Aut(A)$ a subgroup of the automorphism group of $A$.
Denote by $\mathcal{H}^{\mathcal{S}}(A)$ the set of all non zero hereditary
C*-subalgebras of $A$ that are globally $\mathcal{S}$-invariant. We recall the
following definitions from [9]:

1) $\mathcal{S}$ is called weakly ergodic if $A^{\mathcal{S}}$ is trivial.

2) $\mathcal{S\ }$is called topologically transitive if for every $C_{1}$,
$C_{2}\in\mathcal{H}^{\mathcal{S}}(A)$, their product $C_{1}C_{2}=\left\{
\sum_{finite}c_{1}^{i}c_{2}^{i}|c_{1}^{i}\in C_{1},c_{2}^{i}\in C_{2}\right\}
$ is not zero. In the particular case of topological dynamics this condition
is equivalent to the usual topological transitivity of the flow.

In [2] it is noticed that our condition 2) is equivalent to the following:

2') If $x,y\in A$ are not zero, then there is an $s\in\mathcal{S}$ such that
$xs(y)\neq0$.

3) $\mathcal{S}$ is called minimal if $\mathcal{H}^{\mathcal{S}}(A)=\left\{
A\right\}  $.

We caution that in [6] and [16] the notion of minimality is used for compact
actions $(M,G,\delta)$ such that $(M^{G})^{\prime}\cap M=CI$. In this paper, a
group of automorphisms is called minimal if it satisfies condition 3) above.

Obviously 3)$\Longrightarrow$ 2)$\Longrightarrow$ 1). Also, since in the case
of von Neumann algebras, $M$, the $\mathcal{S}$-invariant, hereditary
W*-subalgebras, are of the form $pMp$ where $p$ is a projection in
$M^{\mathcal{S}}$ it follows that all the above conditions are equivalent.
Another situation when all of the above conditions are equivalent for a
C*-dynamical system is when $\mathcal{S}$ is compact [[9], Corollary 2.7.].

In [9] there are also discussed several criteria for checking topological
transitivity. In [2], a seemingly stronger notion than topological
transitivity is introduced, namely the notion of strong topological transitivity:

$\mathcal{S}$ is said to be strongly topologically transitive if for each
finite sequence $\left\{  (x_{i},y_{i})|i=1,2,...,n\right\}  $ of pairs of
elements $x_{i},y_{i}\in A$ for which $\sum x_{i}\otimes y_{i}\neq0$ in the
algebraic tensor product $A\otimes A$, there exists an $s\in\mathcal{S}$ such
that $\sum x_{i}s(y_{i})\neq0$ in $A$.

Further, in [[3], Theorem 5.2.] it is shown that in the case of von Neumann
algebras strong topological transitivity is equivalent with topological
transitivity and hence with the rest of the above conditions.

In what follows we will need the following results from [1]:

\begin{proposition}
Let $(A,\mathcal{S})$ be a C*-dynamical system and $B\subset A$ an
$\mathcal{S}$-invariant C*-subalgebra. Then $\mathcal{H}^{\mathcal{S}%
}(B)=\left\{  C\cap B|C\in\mathcal{H}^{\mathcal{S}}(A)\right\}  $.
\end{proposition}

\begin{proof}
This is [[1], Proposition 2.3.].
\end{proof}

If $A$ is a C*-algebra, we denote by $A_{sa}$ the set of selfadjoint elements
of $A$ and by $(A_{sa})^{m}$ the set of elements in the bidual $A^{\ast\ast}$
of $A$ that can be obtained as strong limits of bounded, monotone increasing
nets from $A_{sa}$ [see also [10], 3.11]. Then we can state [1]:

\begin{proposition}
Let $(A,\mathcal{S)}$ be a C*-dynamical system. Then the following conditions
are equivalent:\newline i) $(A,\mathcal{S)}$ is minimal\newline ii) If
$a\in(A_{sa})^{m}$ is such that $s^{\ast\ast}(a)=a$ for every $s\in
\mathcal{S}$, then $a\in\mathcal{C}I$. Here, $s^{\ast\ast}$ denotes the double
dual of the automorphism $s\in\mathcal{S}$.
\end{proposition}

\begin{proof}
This is [[1], Proposition 4.1.].
\end{proof}

\subsection{Compact group actions on C*-algebras}

Let $(A,G,\delta)$ be a C*-dynamical system with $G$ compact. Denote by
$\widehat{G}$ the set of all equivalence classes of irreducible, unitary
representations of $G$. For each $\pi\in\widehat{G}$, fix a unitary
representation, $u^{\pi}$ in the class $\pi$ and a basis in the Hilbert space
$H_{\pi}$ of $u^{\pi}.$ If $\pi\in\widehat{G}$, denote by $\chi_{\pi
}(g)=d_{\pi}\sum_{i=1}^{d_{\pi}}\overline{u_{ii}^{\pi}(g)}$ the character of
the class $\pi$, where $d_{\pi}$ is the dimension of $H_{\pi}$. For $\pi
\in\widehat{G}$ we consider the following mappings from $A$ into itself:

\begin{center}%
\[
P^{\pi,\delta}(a)=\int_{G}\chi_{\pi}(g)\delta_{g}(a)dg
\]

\[
P_{ij}^{\pi,\delta}(a)=\int_{G}\overline{u_{ji}^{\pi}}(g)\delta_{g}(a)dg
\]

\end{center}

We define the spectral subspaces of the action $\delta$:

\begin{center}%
\[
A_{1}^{\delta}(\pi)=\left\{  a\in A|P^{\pi,\delta}(a)=a\right\}  ,\pi
\in\widehat{G}%
\]

\end{center}

In the particular case when $\pi=\pi_{0}$ is the trivial one dimensional
representation of $G$, $A_{1}^{\delta}(\pi_{0})=A^{G}$ the C*-subalgebra of
fixed elements of the action $\delta$.

As in [8] and [11], the matricial spectral subspaces are defined as follows:

\begin{center}%
\[
A_{2}^{\delta}(u^{\pi})=\left\{  X=[x_{ij}]\in A\otimes B(H_{\pi})|[\delta
_{g}(x_{ij})]=X(1\otimes u^{\pi}(g))\right\}
\]

\end{center}

Notice that $A_{2}^{\delta}(u^{\pi})$ depends on the representation $u^{\pi}$
but for two equivalent representations, $A_{2}^{\delta}(u^{\pi})$ are
spatially isomorphic. Obviously, $A_{2}^{\delta}(u^{\pi})A_{2}^{\delta}%
(u^{\pi})^{\ast}$ is a two sided ideal of $A^{G}\otimes B(H_{\pi})$ and
$A_{2}^{\delta}(u^{\pi})^{\ast}A_{2}^{\delta}(u^{\pi})$ is a two sided ideal
of $(A\otimes B(H_{\pi}))^{\delta\otimes ad(u^{\pi})}$. The proof of the
following remark is straightforward:

\begin{remark}
Let $(A,G,\delta)$ be a C*-dynamical system with $G$ compact and $s\in Aut(A)$
be such that $s\delta_{g}=\delta_{g}s$ for every $g\in G$. Then $s(A_{1}%
^{\delta}(\pi))\subset A_{1}^{\delta}(\pi)$ and $(s\otimes\iota)(A_{2}%
^{\delta}(u^{\pi}))\subset A_{2}^{\delta}(u^{\pi})$ for every $\pi\in
\widehat{G}$. Here $\iota$ stands for the identity automorphism of $B(H_{\pi
})$.
\end{remark}

We will use the following results from [11]:

\begin{lemma}
i) $\sum_{\pi\in\widehat{G}}A_{1}^{\delta}(\pi)$ is dense in $A$\newline ii)
$A_{2}^{\delta}(u^{\pi})=\left\{  [P_{ij}^{\pi,\delta}(a)]|a\in A\right\}  $
\end{lemma}

\begin{proof}
i) This is [[11], Lemma 2.3.]

ii) This is [[11], Lemma 2.2.]
\end{proof}

We also need the following known result:

\begin{lemma}
Let $(C,G,\delta)$ be a C*-dynamical system with $G$ compact. Then every
approximate unit of the fixed point algebra $C^{G}$ is an approximate unit of
$C$.
\end{lemma}

\begin{proof}
See for instance [[5], Lemma 2.7.] for the more general case of compact
quantum group actions.
\end{proof}

Finally, we recall that a C*-dynamical system $(A,G,\delta)$ with $G$ compact
is called saturated if the closed, two sided ideal of the crossed product,
$A\rtimes_{\delta}G$, generated by $\chi_{\pi_{0}}$ equals the crossed
product. In this definition we used the known fact that every character,
$\chi_{\pi}$, of $G$ is an element of the multiplier algebra $M(A\rtimes
_{\delta}G)$ of the crossed product [[8], [11]]. If the system is saturated
then, the crossed product is strongly Morita equivalent, in the sense of
Rieffel, with the fixed point algebra, $A^{G}$ [[13], page 236]. Then, we have
[[11], Corollary 3.5.]:

\begin{lemma}
Let $(C,G,\delta)$ be a C*-dynamical system with $G$ compact. Then the
following conditions are equivalent:\newline i) The system is
saturated\newline ii) The two sided ideal $C_{2}^{\delta}(u^{\pi})^{\ast}%
C_{2}^{\delta}(u^{\pi})$ is dense in $(C\otimes B(H_{\pi}))^{\delta\otimes
ad(u^{\pi})}$ for every $\pi\in\widehat{G}$, where $(C\otimes B(H_{\pi
}))^{\delta\otimes ad(u^{\pi})}$ is the fixed point algebra of $C\otimes
B(H_{\pi})$ for the action $\delta\otimes ad(u^{\pi})$ of $G$.
\end{lemma}

\section{Galois correspondence}

In this section we will prove our main results and give examples and counterexamples.

Let $(A,G,\delta)$ be a C*-dynamical system with $G$ compact. Let $B\subset A$
be $G$-invariant C*-subalgebra such that $A^{G}\subset B.$

\begin{lemma}
If $(B,G,\delta)$ is saturated and if $\delta|_{B}$ is faithful, then $B=A.$
\end{lemma}

\begin{proof}
By Lemma 4 ii), we have to prove that for every $\pi\in\widehat{G}$,
$A_{2}^{\delta}(u^{\pi})\subset B_{2}^{\delta}(u^{\pi}).$ Let $\pi\in
\widehat{G}$ be arbitrary. Since $(B,G,\delta)$ is saturated, by Lemma 6 we
have $\overline{B_{2}^{\delta}(u^{\pi})^{\ast}B_{2}^{\delta}(u^{\pi}%
)}=(B\otimes B(H_{\pi}))^{\delta\otimes ad(u^{\pi})}.$ Since $(A^{G}\otimes
I_{B(H_{\pi})})\subset(B\otimes B(H_{\pi}))^{\delta\otimes ad(u^{\pi})}$, it
follows from Lemma 5 that $(B\otimes B(H_{\pi}))^{\delta\otimes ad(u^{\pi})}$
contains an approximate unit of $A\otimes B(H_{\pi}).$ Since $B_{2}^{\delta
}(u^{\pi})^{\ast}B_{2}^{\delta}(u^{\pi})$ is a dense two sided ideal of
$(B\otimes B(H_{\pi}))^{\delta\otimes ad(u^{\pi})}$, by [[4], Proposition
1.7.2.], this latter C*-algebra contains an approximate unit $\left\{
E_{\lambda}\right\}  \subset B_{2}^{\delta}(u^{\pi})^{\ast}B_{2}^{\delta
}(u^{\pi})$, $E_{\lambda}=\sum_{i=1}^{n_{\lambda}}X_{i,\lambda}^{\ast
}Y_{i,\lambda}$ where $X_{i,\lambda},Y_{i,\lambda}\in B_{2}^{\delta}(u^{\pi
}).$ Let now $X\in A_{2}^{\delta}(u^{\pi}).$ Then $XE_{\lambda}=\sum
XX_{i,\lambda}^{\ast}Y_{i,\lambda}=\sum(XX_{i,\lambda}^{\ast})Y_{i,\lambda}%
\in(A^{G}\otimes B(H_{\pi}))B_{2}^{\delta}(u^{\pi})=B_{2}^{\delta}(u^{\pi}).$
Since $\left\{  E_{\lambda}\right\}  $ is an approximate unit of $A\otimes
B(H_{\pi}),$ it follows that $X=(norm)\lim(XE_{\lambda})\in B_{2}^{\delta
}(u^{\pi}).$ Therefore $B=A.$
\end{proof}

Let $B$ be a $G$-invariant C*-subalgebra of $A$ such that $A^{G}\subset B$.
Denote $G^{B}=\left\{  g\in G|\delta_{g}(b)=b,b\in B\right\}  $. Then we have:

\begin{remark}
i) $G^{B}$ is a closed, normal subgroup of $G$\newline ii) The quotient action
$\delta^{\bullet}$ of $G/G^{B}$ on $B$ is faithful.
\end{remark}

\begin{proof}
Straightforward.
\end{proof}

\begin{corollary}
Let $(A,G,\delta)$ be a C*-dynamical system. If $A^{G}\subset B\subset A$ and
$B$ is a $G$-invariant C*-subalgebra such that $(B,G/G^{B},\delta^{\cdot})$ is
saturated, where $\delta^{\bullet}$ is the quotient action, then $B=A^{G^{B}%
}.$
\end{corollary}

\begin{proof}
By Remark 8 ii) the quotient action $\delta^{\bullet}$ of $G/G^{B}$ on $B$ is
faithful and therefore, if we apply Lemma 7 to $G/G^{B}$ instead of $G,$ we
get the desired result.
\end{proof}

\begin{remark}
As we have noticed in the proof of the previous Lemma, if $B$ is a
$G$-invariant C*-subalgebra of $A$ and $\pi\in\widehat{G}$ is such that
$B_{1}^{\delta}(\pi)\neq(0)$ and hence $B_{2}^{\delta}(u^{\pi})\neq(0)$, it
follows that the ideal $\overline{B_{2}^{\delta}(u^{\pi})^{\ast}B_{2}^{\delta
}(u^{\pi})}$ contains an approximate unit $\left\{  E_{\lambda}\right\}  $ of
the form $E_{\lambda}=\sum_{i=1}^{n_{\lambda}}X_{i,\lambda}^{\ast}%
Y_{i,\lambda}$ where $X_{i,\lambda},Y_{i,\lambda}\in B_{2}^{\delta}(u^{\pi})$.
\end{remark}

\begin{lemma}
Let $A$ be a C*-algebra and $B\subset A$ a C*-subalgebra If $\mathcal{S}%
\subset Aut(A)$ acts minimally on $A$ and leaves $B$ globally invariant, then
$\mathcal{S}$ acts minimally on $B$.
\end{lemma}

\begin{proof}
Indeed, by Proposition 1 every $\mathcal{S}$-invariant hereditary
C*-subalgebra of $B$, $C$, is the intersection of $B$ with an $\mathcal{S}%
$-invariant hereditary subalgebra of $A.$Since $\mathcal{S}$ acts minimally on
$A$ it follows that $C=B$.
\end{proof}

In what follows, we will need the following result from [12]:

\begin{proposition}
Let $(A,G,\delta)$ be a dynamical system with $G$ compact. Assume that the
action $\delta$ is faithful and that there is a subgroup $\mathcal{S}$ of
$Aut_{\delta}(A)$ which acts minimally on $A.$ Then $A_{1}^{\delta}(\pi
)\neq(0)$ for every $\pi\in\widehat{G}$.
\end{proposition}

\begin{proof}
This is [[12], Proposition 5.2.].
\end{proof}

The next lemma provides a class of C*-dynamical systems $(A,G.\delta)$ that
are saturated. We will denote by $Aut_{\delta}(A)$ the subgroup of the group
of all automorphisms of $A$ consisting of all automorphisms that commute with
$\delta_{g}$ for all $g\in G$.

\begin{lemma}
Let $(A,G,\delta)$ be a dynamical system with $G$ compact. Assume that the
action $\delta$ is faithful and that there is a subgroup $\mathcal{S}$ of
$Aut_{\delta}(A)$ which acts minimally on $A.$Then the system is saturated.
\end{lemma}

\begin{proof}
We will prove that $\overline{A_{2}^{\delta}(u^{\pi})^{\ast}A_{2}^{\delta
}(u^{\pi})}=A\otimes B(H_{\pi})^{\alpha\otimes adu^{\pi}}$ for every $\pi
\in\widehat{G}$ and the result will follow from Lemma 6. Notice first that
according to Proposition 12, $A_{1}^{\delta}(\pi)\neq(0)$ for every $\pi
\in\widehat{G}$ and hence $A_{2}^{\delta}(u^{\pi})\neq(0)$ for every $\pi
\in\widehat{G}$. Let $\pi\in\widehat{G}$ be arbitrary. Applying Remark 10, let
$E_{\lambda}=\sum_{i=1}^{n_{\lambda}}X_{i,\lambda}^{\ast}Y_{i,\lambda}$, where
$X_{i,\lambda},Y_{i,\lambda}\in A_{2}^{\delta}(u^{\pi})$, be an increasing
approximate unit of $\overline{A_{2}^{\delta}(u^{\pi})^{\ast}A_{2}^{\delta
}(u^{\pi})}$. Then $E_{\lambda}\nearrow E$ in the strong operator topology,
where $E$ is the unit of the von Neumann algebra generated by $\overline
{A_{2}^{\delta}(u^{\pi})^{\ast}A_{2}^{\delta}(u^{\pi})}$ in $A^{\ast\ast
}\otimes B(H_{\pi})$, where $A^{\ast\ast}$ is the second dual of $A$. Let $H$
be the Hilbert space of the universal representation of $A$ so that
$A^{\ast\ast}\subset B(H)$. Since every $s\in\mathcal{S}$ commutes with every
$\delta_{g},$ $g\in G,$ by Remark 3 it follows that $A_{2}^{\delta}(u^{\pi
})^{\ast}A_{2}^{\delta}(u^{\pi})$ and its weak closure in $A^{\ast\ast}\otimes
B(H_{\pi})$ is globally invariant under the automorphisms $\left\{
s^{\ast\ast}\otimes\iota|s\in\mathcal{S}\right\}  $ where $s^{\ast\ast}$ is
the double dual of $s$ and $\iota$ is the identity automorphism of $B(H_{\pi
})$. This means, in particular, that $(s^{\ast\ast}\otimes\iota)(E)=E$ for
every $s\in\mathcal{S}$. If we write $E_{\lambda}=\left[  E_{ij}^{\lambda
}\right]  $ $i,j=1,...,d_{\pi}$ as a matrix with entries in $A$ and $E=\left[
E_{ij}\right]  $, $E_{ij}\in A^{\ast\ast},$ then $s^{\ast\ast}(E_{ij})=E_{ij}$
for every $s\in\mathcal{S}$ and all $i,j$. Since $E$ is a projection, it is in
particular a positive operator which is a strong limit of elements of
$A\otimes B(H_{\pi})$ Therefore every diagonal entry, $E_{ii}$, of $E$, is a
positive operator which is the strong limit of an increasing net of positive
elements of $A$, so $E_{ii}\in A^{m}$. By Proposition 2 it follows that there
are scalars $\mu_{ii}$ such that $E_{ii}=\mu_{ii}I$ where $I$ is the unit of
$B(H)$. Now, $H\otimes H_{\pi}\simeq\oplus_{i=1}^{i=d_{\pi}}H_{i}$ where
$H_{i}=H$ for all $i=1,...d_{\pi}$ with $d_{\pi}$ the dimension of $H_{\pi}$.
Let $\zeta=\oplus\zeta_{i}\in H\otimes H_{\pi}$ with $\zeta_{i_{0}}%
=\zeta_{j_{0}}=\xi\in H$ and $\zeta_{i}=0$ if $j_{0}\neq i\neq i_{0}$. Then we have%

\begin{gather*}
(E_{\lambda}\zeta,\zeta)=(E_{i_{0}i_{0}}^{\lambda}\xi,\xi)+(E_{i_{0}j_{0}%
}^{\lambda}\xi,\xi)+(E_{i_{0}j_{0}}^{\lambda\ast}\xi,\xi)+(E_{j_{0}j_{0}%
}^{\lambda}\xi,\xi)=\\
=((E_{i_{0}j_{0}}^{\lambda}+E_{i_{0}j_{0}}^{\lambda\ast}+E_{i_{0}i_{0}%
}^{\lambda}+E_{j_{0}j_{0}}^{\lambda})\xi,\xi)
\end{gather*}

Since $E_{\lambda}\nearrow E$, it follows that $(E_{i_{0}j_{0}}^{\lambda
}+E_{i_{0}j_{0}}^{\lambda\ast}+E_{i_{0}i_{0}}^{\lambda}+E_{j_{0}j_{0}%
}^{\lambda})\nearrow(E_{i_{0}j_{0}}+E_{i_{0}j_{0}}^{\ast}+\mu_{i_{0}i_{0}%
}I+\mu_{j_{0}j_{0}}I)$ in the weak operator topology and, since $\left\{
E_{\lambda}\right\}  $ is norm bounded, in the strong operator toplogy. Hence,
$E_{i_{0}j_{0}}+E_{i_{0}j_{0}}^{\ast}+\mu_{i_{0}i_{0}}I+\mu_{j_{0}j_{0}}I\in
A^{m}$. As we noticed before, $E_{i_{0}j_{0}}+E_{i_{0}j_{0}}^{\ast}+\mu
_{i_{0}i_{0}}I+\mu_{j_{0}j_{0}}I$ is $s^{\ast\ast}$-invariant for every
$s\in\mathcal{S}$ and therefore, by Proposition 2 it is a scalar multiple of
the identity. Hence $E_{i_{0}j_{0}}+E_{i_{0}j_{0}}^{\ast}$ is a scalar
multiple of the identity.

Similarly, considering $\zeta=\oplus\zeta_{i}\in H\otimes H_{\pi}$ with
$\zeta_{i_{0}}=\sqrt{-1}\xi$, $\zeta_{j_{0}}=-\xi$, $\xi\in H$ and $\zeta
_{i}=0$ if $j_{0}\neq i\neq i_{0}$ we infer that $E_{i_{0}j_{0}}-E_{i_{0}%
j_{0}}^{\ast}$ is a scalar multiple of the identity. Hence there are scalars
$\mu_{ij}$ such that $E_{ij}=\mu_{ij}I$, so all entries of $E$ are scalar
multiples of the identity. Since $E$ is an element of the weak closure of
$A\otimes B(H_{\pi})^{\delta\otimes ad(u^{\pi})}$ it follows that $E$
intertwines $u^{\pi}$ with itself and therefore, since $u^{\pi}$ is
irreducible, we have $E=I$ and we are done.
\end{proof}

We can now prove our main result:

\begin{theorem}
Let $(A,G,\delta)$ be a dynamical system with $G$ compact. Assume that there
is a subgroup $\mathcal{S}$ of $Aut_{\delta}(A)$ which acts minimally on $A.$
If $A^{G}\subset B\subset A$ and $B$ is a $G$ and $\mathcal{S}$ globally
invariant C*-subalgebra, then $B=A^{G^{B}}$. Conversely, if $G_{0}\subset G$
is a closed, normal subgroup, then $B=A^{G_{0}}$ is a $G$ and $\mathcal{S}$
invariant C*-subalgebra such that $A^{G}\subset B\subset A$.
\end{theorem}

\begin{proof}
It is immediate to see that the quotient action $\delta^{\bullet}$ acts
faithfully on $B$. By Lemma 11, $\mathcal{S}$ acts minimally on $B.$ By Lemma
13, the system $(B,\delta^{\bullet},G/G^{B})$ is saturated. By Corollary 1,
$B=A^{G^{B}}$ and we are done. The converse is easily checked.
\end{proof}

Notice that for W*-dynamical systems the proof of the above Theorem 14 is
simpler since the discussion about lower semicontinuous elements in the bidual
$A^{\ast\ast}$ is not necessary.

A simple example of a C*-dynamical system $(A,G,\delta)$ with $G$ compact
satisfying the hypotheses of Theorem 14 is the following:

\begin{example}
Let $G$ be a compact group and $C(G)$ the C*-algebra of continuous functions
on $G$. Denote by $\lambda$ the action of $G$ on $C(G)$ by left translations
and by $\rho$ the action by right translations. Let $H$ be a Hilbert space and
$\mathcal{K}(H)$ the algebra of compact operators on $H$. Let $A=C(G)\otimes
\mathcal{K}(H)$ and $\delta_{g}=\lambda_{g}\otimes\iota,g\in G$. Then the
subgroup $\mathcal{S}\subset Aut_{\delta}(A)$ generated by $\left\{  \rho
_{g}\otimes ad(u)|g\in G,u\in\widetilde{\mathcal{K}(H)},unitary\right\}  $
where $\rho_{g}$ is the right translation by $g\in G$, acts minimally on $A$.
Here $\widetilde{\mathcal{K}(H)}$ denotes the C*-algebra obtained from
$\mathcal{K}(H)$ by adjoining a unit if $H$ is infinite dimensional.
\end{example}

The next result provides a class of examples of C*-dynamical systems
$(A,G,\delta)$ that satisfy the hypotheses of Theorem 14.

\begin{theorem}
Let $(A,G,\delta)$ be a C*-dynamical system with $G$ compact abelian. Assume
that the fixed point algebra $A^{G}$ is simple. If $B$ is a $G$-invariant
C*-subalgebra such that $A^{G}\subset B\subset A$, then $B=A^{G^{B}}$.
\end{theorem}

\begin{proof}
Denote by $\widetilde{A^{G}}$ the C*-algebra obtained from $A^{G}$ by
adjoining a unit. We will show that the subgroup $\mathcal{S}\subset Aut(A)$
generated by $\delta_{G}=\left\{  \delta_{g}|g\in G\right\}  $ and $\left\{
ad(u)|u\in\widetilde{A^{G}},unitary\right\}  $ is minimal. Since $G$ is
abelian and $ad(u)$, $u\in\widetilde{A^{G}}$ commute with $\delta_{g},g\in G$,
we have that $\mathcal{S}\subset Aut_{\delta}(A)$. We prove next that
$\mathcal{S}$ acts minimally on $A$. Let $C\in\mathcal{H}^{\mathcal{S}}(A)$.
Then if $L=\overline{AC}$, we have that $L$ is an $\mathcal{S}$-invariant, in
particular $G$-invariant, closed, left ideal of $A$ and $C=L\cap L^{\ast}$. We
show that $L=A$ and hence $C=A$. Since $L$ is $\mathcal{S}$-invariant, it
follows that $L^{G}=\left\{  \int_{G}\delta_{g}(l)dg|l\in L\right\}  \subset
L\cap A^{G}$ is a left ideal of $A^{G}$. Since $L$ is $ad(u)$-invariant for
every $u\in\widetilde{A^{G}}$, unitary, we have:%
\[
L^{G}u=uad(u^{\ast})(L^{G})\subset L^{G},u\in\widetilde{A^{G}}\text{ unitary.}%
\]
Therefore $L^{G}$ is a two sided ideal of $A^{G}$. Since $A^{G}$ is simple, it
follows that $L^{G}=A^{G}$ and thus by Lemma 5, $L^{G}$ and so $L$ contains an
approximate unit of $A$. Hence $L=A$ and therefore $C=L\cap L^{\ast}=A$.
Therefore $\mathcal{S}\subset Aut_{\delta}(A)$ is minimal and the conclusion
follows from Theorem 14.
\end{proof}

An example of C*-dynamical system satisfying the hypotheses of Theorem 16 can
be constructed as follows:

\begin{example}
Let $(C,G,\lambda)$ be a C*-dynamical system with $G$ compact abelian. Assume
that $\lambda$ is weakly ergodic, and therefore minimal, by [[9], Corollary
2.7.]. Let $H$ be a Hilbert space. Let $A=C\otimes\mathcal{K}(H)$, where
$\mathcal{K}(H)$ is the algebra of compact operators on $H$ and $\delta
_{g}=\lambda_{g}\otimes\iota,g\in G$ where $\iota$ is the trivial automorphism
of $\mathcal{K}(H)$. Then $(A,G,\delta)$ satisfies the hypotheses of Theorem 16.
\end{example}

\begin{proof}
Straightforward.
\end{proof}

The next example shows that the conclusion of Theorem 14 may fail if the
minimality condition on $\mathcal{S}$ is replaced with a weaker ergodicity
condition such as topological transitivity, or even with strong topological transitivity.

\begin{example}
Let $G$ be a compact abelian group and $H$ an infinit dimensional Hilbert
space. Denote by $\tau$ the action of $G$, by translations, on $C(G)$, the
C*-algebra of continuous functions on $G$. Let $A=C(G)\otimes\widetilde
{\mathcal{K}(H)}$, where $\widetilde{\mathcal{K}(H)}$ is the subalgebra of
$B(H)$ generated by $\mathcal{K}(H)$ and the unit $I\in B(H)$. Let $\delta
_{g}=\tau_{g}\otimes\iota,g\in G$, where $\iota$ is the identity automorphism
of $\mathcal{K}(H)$. Consider the system $(A,G,\delta)$. Clearly,
$A^{G}=I_{C(G)}\otimes\widetilde{\mathcal{K}(H)}$. We will prove the following
two facts:\newline i) $Aut_{\delta}(A)$ contains a subgroup $\mathcal{S}$
which acts strongly topologically transitively on $A$\newline ii) There is a
$G$ and $\mathcal{S}$-invariant C*subalgebra $B$ such that $A^{G}\subset
B\subset A$ and $B\neq A^{G^{B}}$.
\end{example}

\begin{proof}
i) Let $\mathcal{S}=\left\{  \tau_{g}\otimes ad(u)|g\in G,u\in\widetilde
{\mathcal{K}(H)},unitary\right\}  \subset Aut(A)$. Obviously, every element
$s\in\mathcal{S}$ commutes with all $\delta_{g}=\tau_{g}\otimes\iota,g\in G$.
We prove next that $\mathcal{S}$ acts ergodically on the von Neumann algebra
$L^{\infty}(G)\otimes B(H)$ and then applying [[9], Theorem 2.2. i)]
(respectively, [[3], Corollary 5.3.]) it will follow that $\mathcal{S}$ acts
topologically transitively (strongly topologically transitively) on $A$.
Notice first that $\tau_{g}$ is implemented by the unitary operator
$\lambda_{g}\in B(L^{2}(G))$ of translation by $g$. Hence the fixed point
algebra $(B(L^{2}(G))\otimes B(H))^{\mathcal{S}}$ is the commutant $(C^{\ast
}(G)^{^{\prime\prime}}\otimes B(H))^{^{\prime}}=C^{\ast}(G)^{^{\prime\prime}%
}\otimes\mathcal{C}I$ where $C^{\ast}(G)$ is the group C*-algebra of $G$.
Since $C^{\ast}(G)^{^{\prime\prime}}\cap L^{\infty}(G)=\mathcal{C}I$, i) is
proven.\newline ii) Let $B\subset A$ be the C*-subalgebra generated by
$C(G)\otimes\mathcal{K}(H)$ and $I_{C(G)}\otimes I_{B(H)}$. Then, $B$ is
obviously $G$ and $\mathcal{S}$ invariant and $A^{G}=I_{C(G)}\otimes
\widetilde{\mathcal{K}(H)}\subset B$. Clearly, $G^{B}=\left\{  g\in
G|\delta_{g}(b)=b,b\in B\right\}  =\left\{  e\right\}  $ where $e$ is the
identity element of $G$. If we show that $B\neq A$, ii) is proven. Let $f\in
C(G)$ be a non constant function. Then there are $g_{1},g_{2}\in G$ such that
$f(g_{1})\neq f(g_{2})$. We claim that $f\otimes I\notin B$. Assume to the
contrary that $f\in B$. then there is a function $\Phi:G\rightarrow
\mathcal{K}(H)$ and a scalar $\mu$ such that $f(g)\otimes I=\Phi(g)+\mu I$. In
particular $(f(g_{1})-f(g_{2}))I=\Phi(g_{1})-\Phi(g_{2})\in\mathcal{K}(H)$,
which is a contradiction since $f(g_{1})-f(g_{2})$ is a non zero scalar.
\end{proof}

\bigskip

\bigskip

{\Large Acknowledgement}

\bigskip

The author would like to thank the referee for several suggestions that
contributed to an improved presentation of this paper.

\bigskip

\bigskip

{\Large References}

\bigskip

1. D. Avitzour, Noncommutative topological dynamics I, Trans. Amer. Math. Soc.
282 (1984), 109-119.

2. O. Bratteli, G. A. Elliott and D. W. Robinson, Strong topological
transitivity and C*-dynamical systems, J. Math. Soc. Japan 37 (1985), 115-133.

3. O. Bratteli, G. A. Elliott, D. E. Evans and A Kishimoto, Quasi-product
actions of a compact abelian group on a C*-algebra, Tohoku Math. J. 41 (1989), 133-161.

4. J. Dixmier, Les C*-alg\`{e}bres et leurs repr\'{e}sentations,
Gauthier-Villars, Paris, 1964.

5. R. Dumitru and C. Peligrad, Compact quantum group actions on C*-algebras
and invariant derivations, Proc. Amer. Math. Soc. 135 (2007), 3977-3984.

6. M. Izumi, R. Longo and S. Popa, A Galois correspondence for compact groups
of automorphisms of von Neumann algebras with a generalization to Kac
algebras, J. Funct. Anal. 155 (1998),25-63.

7. A. Kishimoto, Remarks on compact automorphism group of a certain von
Neumann algebra, Publ. R.I.M.S., Kyoto Univ., 13 (1977), 573-581.

8. M. B. Landstad, Algebras of spherical functions associated with a covariant
system over a compact group, Math. Scand. 47 (1980), 137-149.

9. R. Longo and C. Peligrad, Noncommutative topological dynamcs and compact
actions on C*-algebras, J. Funct. Anal. 58 (1984), 157-174.

10. G. K. Pedersen, C*-algebras and their automorphism groups, Academic Press,
London New York San Francisco, 1979.

11. C. Peligrad, Locally compact group actions on C*-algebras and compact
subgroups, J. Funct. Anal. 76 (1988), 126-139.

12. C. Peligrad, Compact actions commuting with ergodic actions and
applications to crossed products, Trans. Amer. Math. Soc., 331 (1992), 825-836.

13. N. C. Phillips, Equivariant K-theory and freeness of group actions on
C*-algebras, Lecture Notes in Mathematics, Springer-Verlag, Berlin Heidelberg
New York London Paris Tokyo, 1274 (1987).

14. J. E. Roberts, Cross products of a von Neumann algebra by group duals,
Symposia Math., 20 (1976), 335-363.

15. M. Takesaki, Fourier analysis of compact automorphism groups (an
application of Tannaka duality theorem), in Coll. Internat. du CNRS 274 (1979).

16. R. Tomatsu, A Galois correspondence for compact quantum group actions,
J.reine angew. Math. 633 (2009), 165-182.

\end{document}